\documentclass[12pt]{article}
\usepackage[utf8]{inputenc}
\usepackage{amsmath}
\usepackage{amssymb}
\usepackage{amsthm}
\usepackage{fullpage}
\usepackage{graphicx}
\usepackage{hyperref}
\usepackage{color}
\usepackage[sc]{mathpazo}
\usepackage[T1]{fontenc}
\usepackage[margin=10pt,font=small]{caption}
\usepackage{booktabs}

\newtheorem{thm}{Theorem}[section]
\newtheorem{prop}[thm]{Proposition}
\newtheorem{lem}[thm]{Lemma}
\newtheorem{coro}[thm]{Corollary}
\theoremstyle{remark}
\newtheorem{rmk}{Remark}

\newcommand{\tdef}[1]{\textcolor{blue}{\emph{#1}}}
\newcommand{\horiz}{\begin{center}\rule{0.3\textwidth}{0.5pt}\end{center}}

\newcommand{\id}{\operatorname{id}}
\newcommand{\cyc}{\operatorname{cyc}}
\newcommand{\hook}{\mathfrak{h}}
\newcommand{\sk}{\operatorname{sk}}
\newcommand{\ct}{\operatorname{ct}}

\author{
  Yichao Chen \\ 
  {\small School of Mathematics and Physics} \\ 
  {\small SuZhou University of Science and Technolgy, 215009 SuZhou, China} \\ 
  {\small \url{chengraph@163.com}}
  \and 
  Wenjie Fang\thanks{\textit{Recherche en esprit libre}, \textit{i.e.}, not supported by any funding with precise predefined goals.} \\
  {\small LIGM, Université Gustave Eiffel, CNRS, ESIEE Paris, 77454 Marne-la-Vallée, France} \\
  {\small \url{wenjie.fang@u-pem.fr}}
}

\title{A character approach to directed genus distribution of graphs: the bipartite single-black-vertex case}

\begin{document}

\maketitle

\abstract{Given an Eulerian digraph, we consider the genus distribution of its face-oriented embeddings. We prove that such distribution is log-concave for two families of Eulerian digraphs, thus giving a positive answer for these families to a question asked in Bonnington, Conder, Morton and McKenna (2002). Our proof uses real-rooted polynomials and the representation theory of the symmetric group $\mathbb{S}_n$. The result is also extended to some factorizations of the identity in $\mathbb{S}_n$ that are rotation systems of some families of one-face constellations.}

\horiz

In topological graph theory, we consider embeddings of graphs on surfaces. For a fixed graph, it may be embedded into surfaces of different genera, even when restricted to cellular embeddings. There have been some effort to asymptotically count graphs that can be embedded in a given surface, with consequences in probabilities \cite{graphfive, cubic-asympt}. We are here interested in the dual question: for a fixed graph, how many ways it can be cellularly embedded into surfaces of different genera?

In this article, we only consider orientable closed surfaces. For an edge-labeled graph $G$, its \emph{genus polynomial} $P_G(x)$ is defined by
\[
  \Gamma_G(x) = \sum_{M \text{ cellular embedding of } G} x^{g(M)},
\]
with $g(M)$ the genus of $M$. The coefficients of the genus polynomial of $G$ give the \emph{genus distribution} of $G$. It was conjectured by Gross, Robbins and Tucker in \cite{bouquet-genus} that the genus distribution of every graph is log-concave, and it was confirmed on bouquets of circles with a formula in \cite{jackson-character} using characters in the symmetric group. Stahl provided more confirmations in \cite{stahl-genus} by proving that the genus polynomials of several families of graphs are real-rooted, a stronger property than log-concavity. Based on such results, Stahl conjectured in \cite{stahl-genus} that genus polynomials for graphs are all real-rooted. An attempt to refute this conjecture was given in \cite{counter-example-undirected}, but was found to be erroneous in \cite{defend-stahl}. Later a valid counter-example was given in \cite{stahl-refute}. Nevertheless, Stahl's work still suggests real-rootedness as a possible angle of attack for log-concavity in some cases. The original conjecture in \cite{bouquet-genus}, however, has received further confirmations (see, \textit{e.g.}, \cite{log-concave-genus}), but remains open.

The genus distribution problem is later extended to directed graphs, or \emph{digraphs} for short. More precisely, we are interested in \emph{face-directed embeddings} of Eulerian digraphs, \textit{i.e.}, cellular embeddings where each face has all its adjacent edges oriented in the same way. Precise definitions of related notions are postponed to Section~\ref{sec:prelim}. Such embeddings were first considered by Bonnington, Conder, Morton and McKenna in \cite{embed-digraph}. Similarly, for an edge-labeled Eulerian digraph $D$, its genus polynomial $P_D(x)$ is defined by
\[
  \Gamma_D(x) = \sum_{M \text{ face-directed embedding of } D} x^{g(M)}.
\]
At the end of \cite{embed-digraph}, it was asked whether genus polynomials of Eulerian digraphs are always log-concave, or at least unimodal. This question mirrors the conjecture in \cite{bouquet-genus}, and was confirmed on several families of Eulerian digraphs, most notably on 4-regular outerplanar digraphs \cite{four-reg-genus}.

The results in this article follow the same direction, providing more confirmation on the log-concavity of genus distribution of Eulerian digraphs. More precisely, we prove the following result.

\begin{thm} \label{thm:main-log-concave}
  The genus distributions of the following Eulerian digraphs are log-concave:
  \begin{itemize}
  \item bipartite Eulerian digraphs with only one black vertex;
  \item Eulerian fans (see Section~\ref{sec:other} for definition).
  \end{itemize}
\end{thm}

This result is a combination of Corollary~\ref{coro:dipole} and Proposition~\ref{prop:eulerian-fan}. In fact, we prove the stronger result that the genus polynomials of Eulerian digraphs in Theorem~\ref{thm:main-log-concave} have all their roots real and negative, which implies the log-concavity. For the first family, it is done by first using the representation theory of the symmetric group in the spirit of \cite{bouquet-genus} to obtain an explicit expression of related genus polynomials (Theorem~\ref{thm:p-expr}), then proving their real-rootedness using arguments (Proposition~\ref{prop:shift}) adapted from \cite{stanley-factorization,postnikov-stanley}. We then computed their asymptotic genus distribution (Theorem~\ref{thm:asympt-dist}). For the second family, our proof relies on a bijective link to graphs in the first family (see Proposition~\ref{prop:eulerian-fan}).

Some of our results (Theorem~\ref{thm:main} and Theorem~\ref{thm:asympt-dist}) also apply to the genus distribution of factorizations of the identity in the symmetric group into $(m+1)$ permutations, one of them a large cycle and another one with given cycle type. Such factorizations correspond to so-called \emph{one-face $m$-constellations} with given degree profile for vertices of one of the colors.

The rest of the article is organized as follows. In Section~\ref{sec:prelim}, we introduce some notions that we need such as the relevant rotation systems. Then in Section~\ref{sec:chara}, we use characters of the symmetric group $\mathbb{S}_n$ to express the genus polynomials of the Eulerian digraphs that we study, and more generally for some families of factorizations in $\mathbb{S}_n$. These genus polynomials are then proved to be real-rooted with non-positive roots in Section~\ref{sec:realrooted}, leading to log-concave genus distribution. We compute the asymptotic genus distributions in Section~\ref{sec:asympt}. Then in Section~\ref{sec:other} we use graph transformations to transfer genus distribution results in Section~\ref{sec:realrooted} to another family of Eulerian digraphs.

\section{Preliminaries} \label{sec:prelim}

\subsection{Digraphs and maps}

In this work, we consider both undirected graphs (or simply \emph{graphs}) and directed graphs (or \emph{digraphs} for short) with possible loops and multiedges, all coming in \emph{edge-labeled} version. Furthermore, we only consider the \emph{connected} ones. A digraph $D$ is \tdef{Eulerian} if every vertex has equal in-degree and out-degree. A graph or a digraph is \tdef{bipartite} if its vertices can be properly 2-colored, say black and white.

We now consider embeddings of graphs. A \tdef{combinatorial map} (or simply \tdef{map}) on a connected orientable compact surface $\mathbb{S}$ is an embedding $M$ of a graph $G$ into $\mathbb{S}$ that is \tdef{cellular}, \textit{i.e.}, all the connected components in $\mathbb{S} \setminus M$, which are called \tdef{faces}, are topological disks. Maps are defined up to orientation-preserving homeomorphism. We note that the maps we consider are \emph{edge-labeled} but not rooted, in contrast to unlabeled rooted maps in most literature on maps. However, it is known that unlabeled rooted maps with $n$ edges are in $1$-to-$(n-1)!$ correspondence with edge-labeled unrooted maps. Figure~\ref{fig:map-example}(a) shows an example of a map, which is also bipartite. The \tdef{genus} of a map $M$, denoted by $g(M)$, is the genus of the surface $\mathbb{S}$ on which it lives. The genus of a map can be computed using Euler's relation: for a map $M$ with $v$ vertices, $e$ edges and $f$ faces, its genus $g$ satisfies
\begin{equation}
  \label{eq:euler}
  v - e + f = 2 - 2g.
\end{equation}

Similarly, cellular embeddings of a digraph are called \tdef{directed maps}, which are just maps with edges oriented. A directed map is \tdef{face-oriented} if, for every face, the adjacent edges are all oriented in the same direction. Face-directed embeddings of Eulerian digraphs are simply face-oriented maps. Figure~\ref{fig:map-example}(b) shows an example of a face-oriented directed map which is also bipartite.

\begin{figure}[!thbp]
  \centering
  \includegraphics[width=0.8\textwidth,page=2]{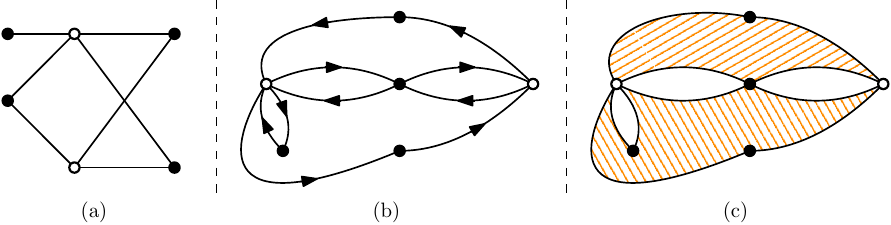}
  \caption{Examples of a bipartite map, a face-oriented bipartite map, and its corresponding Eulerian map (edge labels are omitted)}
  \label{fig:map-example}
\end{figure}

We now introduce another family of maps. A map is \tdef{Eulerian} if its faces can be properly 2-colored. See Figure~\ref{fig:map-example}(c) for an example of Eulerian map that is also bipartite. We note that the word ``Eulerian'' has a different meaning here from that in ``Eulerian digraphs''. Such maps are in bijection with face-oriented directed maps.

\begin{prop} \label{prop:bieulerian}
  The set of face-oriented directed maps with $n$ edges is in bijection with the set of Eulerian maps with $n$ edges.
\end{prop}
\begin{proof}
  Let $\vec{M}$ be a face-oriented directed maps with $n$ edges. We color by white the faces in $\vec{M}$ with edges oriented in the clockwise direction, and by black all others. Now, by forgetting edge orientation, we obtain an Eulerian map $M$, since the faces on the two sides of an edge cannot have the same orientation. To go back from $M$ to $\vec{M}$, we only need to give the edges an orientation according to face color, which is consistent because $M$ is Eulerian.
\end{proof}

An example of this correspondence can be seen in Figure~\ref{fig:map-example}~(b)~and~(c). By Proposition~\ref{prop:bieulerian}, the study of the genus distribution of a given digraph $D$ reduces to the enumeration of Eulerian maps with $D$ as the underlying digraph.

\subsection{Log-concavity and real-rootedness}

For a digraph $D$, we denote by $\mathcal{M}(D)$ the set of Eulerian maps corresponding to face-oriented embeddings of $D$ as indicated in Proposition~\ref{prop:bieulerian}. The \tdef{genus polynomial} of $D$, denoted by $\Gamma_D(x)$, is defined by
\begin{equation}
  \label{eq:genus-poly-def}
  \Gamma_D(x) = \sum_{M \in \mathcal{M}(D)} x^{g(M)}.
\end{equation}
Given a polynomial $P(x)$, we denote by $[x^k]P(x)$ the coefficient of $x^k$ in $P(x)$. It is clear that $\Gamma_D(x)$ encodes all the information of the genus distribution of directed face-oriented embeddings of $D$, since $[x^g]\Gamma_D(x)$ is the number of face-oriented embeddings of $D$ of genus $g$.

We now introduce a few notions for the study of genus distributions. A finite sequence of strictly positive real numbers $A = (a_0, \ldots, a_n)$ is \tdef{log-concave} if $a_k^2 \geq a_{k-1}a_{k+1}$ holds for all $0 < k < n$. Such a sequence is unimodal. We define the \tdef{generating polynomial} of $A$ to be $P_A(x) = \sum_{k=0}^n a_k x^k$. We say that $P_A$ is \tdef{real-rooted} if all roots of $P_A$ are real. The following result is well-known and attributed to Newton.

\begin{lem} \label{lem:real-rooted}
If the generating polynomial $P_A$ of a non-negative sequence $A$ is real-rooted, then $A$ is log-concave.
\end{lem}

In fact, real-rootedness is much stronger than log-concavity, see Lemma~1.1 in \cite{log-concave-survey}. Using the following result from Bender \cite{bender-normal}, real-rootedness can also be used to analyze the asymptotic distribution of some combinatorial sequences.

\begin{lem}[Theorem~2 in \cite{bender-normal}, see also Lemma~2.1 in \cite{log-concave-survey}] \label{lem:rr-dist}
  Let $(X_n)_{n \geq 1}$ be a sequence of random variables, with $X_n$ taking values in $\{0, 1, \ldots, n\}$, and $(P_n(x))_{n \geq 1}$ their generating polynomials defined by $[x^k]P_n(x) = \mathbb{P}[X_n = k]$. Suppose that
  \begin{itemize}
  \item $P_n(x)$ is real-rooted and has only non-positive roots;
  \item $\mathbf{Var}(X_n) \to \infty$.
  \end{itemize}
  Then we have the following convergence to the normal distribution:
  \[
    \frac{X_n - \mathbb{E}[X_n]}{\mathbf{Var}(X_n)^{1/2}} \overset{L^{\infty}}{\longrightarrow} \mathcal{N}(0,1).
  \]
\end{lem}

Readers are referred to \cite{log-concave-survey} for further information on log-concavity and real-rootedness.

Finally, we introduce some concepts needed to describe the class of digraphs in whose genus polynomial we are interested. Let $n$ be a positive integer. An \tdef{integer partition} (or \tdef{partition} for short) $\lambda$ of $n$ is a non-increasing sequence $\lambda = (\lambda_1, \lambda_2, \ldots)$ of positive integers such that $\sum_{i \geq 1} \lambda_i = n$. In this case, we also write $\lambda \vdash n$. It is clear that $\lambda_i$ will eventually be $0$. Non-zero components are called \tdef{parts} of $\lambda$. The \tdef{length} of the partition $\lambda$, denoted by $\ell(\lambda)$, is the number of parts in $\lambda$. The number of parts in $\lambda$ equal to $i$, also called the \tdef{multiplicity} of $i$ in $\lambda$, is denoted by $m_i(\lambda)$.

Given $n > 0$ and a partition $\lambda \vdash n$, we denote by $D_{n, \lambda}$ the bipartite Eulerian digraph with a black vertex $v_\bullet$ of degree $2n$ and $\ell(\lambda)$ white vertices $v_1, v_2, \ldots, v_{\ell(\lambda)}$ such that $v_i$ has degree $2\lambda_i$ for $1 \leq i \leq \ell(\lambda)$. The labels of edges in $D_{n, \lambda}$ are given as follows: the out-going edges of $v_i$ are labeled from $1 + \sum_{j < i} \ell_j$ to $\sum_{j \leq i} \ell_j$ for each $i$, and the out-going edges of $v_\bullet$ are labeled from $n+1$ to $2n$. It is clear that $D_{n, \lambda}$ is unique and strongly connected. It is the genus polynomial $\Gamma_{D_{n, \lambda}}(x)$ of $D_{n, \lambda}$, also denoted by $\Gamma_{n, \lambda}(x)$ for convenience, that we will study.

\subsection{Rotation systems of bi-Eulerian maps}

The study of $\Gamma_{n,\lambda}(x)$ reduces to the enumeration of Eulerian maps in $\mathcal{M}(D_{n,\lambda})$ with respect to their genera. An Eulerian map is \tdef{bi-Eulerian} if it is bipartite and Eulerian. It is clear that the proof of Proposition~\ref{prop:bieulerian} specializes to the bipartite case, thus all maps in $\mathcal{M}(D_{n,\lambda})$ are bi-Eulerian.

Let $M$ be a bi-Eulerian maps of $2n$ edges. We can essentially encoded $M$ by a tuple of permutations in $\mathbb{S}_n$ called \emph{rotation system}. We first relabel $M$ such that out-going edges from black vertices are labeled $(i,\bullet)$ for $i$ from $1$ to $n$, and other edges labeled $(i,\circ)$, such that $(i,\circ)$ is the edge next to $(i,\bullet)$ in clockwise order. There is thus a $(2n)!$-to-$n!$ correspondence from bi-Eulerian maps to such relabeled maps.

Now we encode how edges with labels $(i, \bullet)$ (resp. $(i, \circ)$) surround black (resp. white) vertices in counter-clockwise order by $\sigma_\bullet$ (resp. $\sigma_\circ$). Due to the construction of the new labeling, $\sigma_\bullet$ can also be seen as recording how edges with labels $(i, \circ)$ surround black vertices. Faces in $M$ are either white (edges in clockwise direction) or black (counter-clockwise). We encode by $\phi_\circ$ (resp. $\phi_\bullet$) how edges with labels $(i,\circ)$ surround white (resp. black) faces in counter-clockwise order. We have the property that $\phi_\circ \sigma_\circ \phi_\bullet  \sigma_\bullet = \id_n$, as we can see in Figure~\ref{fig:rotation-system}.

\begin{figure}[!tbp]
  \centering
  \includegraphics[page=1,width=0.3\textwidth]{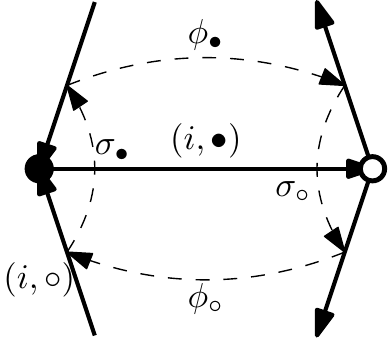}
  \caption{Rotation system and factorization property of embeddings of two-vertex bipartite Eulerian digraph}
  \label{fig:rotation-system}
\end{figure}

$M$ is encoded completely by the triple $(\sigma_\bullet, \sigma_\circ, \phi_\bullet)$, because the information of how all edges surround black vertices is contained in $\sigma_\bullet$ and the convention that $(i,\circ)$ comes next to $(i,\bullet)$ in clockwise order around their adjacent black vertex, while the same information for white vertices is similarly contained in $\sigma_\circ$ and $\sigma_\bullet^{-1} \phi_\bullet^{-1}$. We thus have the following for rotation systems of bi-Eulerian maps. We denote by $\mathbb{S}_n$ the \tdef{symmetric group} of order $n$.

\begin{prop} \label{prop:bieulerian-rot}
  Bi-Eulerian maps with $2n$ edges are in $(2n)!$-to-$n!$ correspondence with tuples $(\phi_\circ, \sigma_\circ, \phi_\bullet, \sigma_\bullet)$ of permutations in $\mathbb{S}_n$ that are transitive and satisfy $\phi_\circ \sigma_\circ \phi_\bullet \sigma_\bullet = \id_n$. We call such tuples \tdef{rotation systems} of bi-Eulerian maps.
\end{prop}

The genus of $M$ can be computed from any of its rotation systems $(\phi_\circ, \sigma_\circ, \phi_\bullet, \sigma_\bullet)$. For $\sigma \in \mathbb{S}_n$, we denote by $\cyc(\sigma)$ the number of cycles in $\sigma$. Then by construction, the number of vertices in $M$ is $\cyc(\sigma_\bullet) + \cyc(\sigma_\circ)$, while the number of face is $\cyc(\phi_\bullet) + \cyc(\phi_\circ)$. By Euler's relation \eqref{eq:euler}, we have
\begin{equation}
  \label{eq:cyc-genus}
  \cyc(\sigma_\bullet) + \cyc(\sigma_\circ) + \cyc(\phi_{ccw}) + \cyc(\phi_{cw}) = 2n + 2 - 2g(M).
\end{equation}

Given $\sigma \in S_n$, the \tdef{cycle type} of $\sigma$ is the partition whose parts consist of the length of each cycle in $\sigma$. We denote by $C(\lambda)$ the set of permutations with cycle type $\lambda \vdash n$, which is also the \tdef{conjugacy class} of $S_n$ associated to $\lambda$. Since $D_{n, \lambda}$ is the unique bipartite Eulerian digraph with $[n]$ and $\lambda$ as the degree profiles of black and white vertices respectively, rotation systems $(\phi_\circ, \sigma_\circ, \phi_\bullet, \sigma_\bullet)$ for bi-Eulerian maps in $\mathcal{M}(D_{n,\lambda})$ are exactly those with $\sigma_\bullet \in C([n])$ and $\sigma_\circ \in C(\lambda)$. Such rotation system is always transitive, as $\sigma_\bullet$ is an $n$-cycle. We thus have the following.

\begin{coro} \label{coro:bieulerian-rot}
  Bi-Eulerian maps in $\mathcal{M}(D_{n,\lambda})$ are in $\left( \left( \prod_{i=1}^{\ell(\lambda)} \lambda_i! \right)^2 \prod_{i \geq 1} m_i(\lambda)!\right)$-to-$n!$ correspondence with rotation systems $(\phi_\circ, \sigma_\circ, \phi_\bullet, \sigma_\bullet)$ with $\sigma_\bullet \in C([n])$ and $\sigma_\circ \in C(\lambda)$.
\end{coro}
\begin{proof}
  We first count the number of ways to relabel $D_{n,\lambda}$. For each white vertex of degree $d$, we choose $d$ labels for its out-going edges and $d$ labels for incoming ones. However, if there are $k$ vertices of the same degree, permuting their labels give the same labeling. Therefore, the total number of ways to relabel $D_{n,\lambda}$ is
  \[
    \frac{(2n)!}{\left( \prod_{i=1}^{\ell(\lambda)} \lambda_i! \right)^2 \prod_{i \geq 1} m_i(\lambda)!}.
  \]
  We thus obtain our claim by combining the two observations with Proposition~\ref{prop:bieulerian-rot}.
\end{proof}

From Corollay~\ref{prop:bieulerian-rot}, we have the following expression of $\Gamma_{n,\lambda}(t)$:
\begin{equation}
  \label{eq:gpoly-fact}
  \Gamma_{n,\lambda}(x) = \frac1{n!} \left(\prod_{i=1}^{\ell(\lambda)} \lambda_i!\right)^2 \cdot \left( \prod_{i \geq 1} m_i(\lambda)! \right) \cdot \sum_{\substack{\phi_\circ \sigma_\circ \phi_\bullet \sigma_\bullet = \id_{n} \\ \sigma_\bullet \in C([n]), \sigma_\circ \in C(\lambda)}} x^{n - \frac{1}{2}(\ell(\lambda) + \cyc(\phi_\circ) + \cyc(\phi_\bullet))}.
\end{equation}

Rotations systems of bi-Eulerian maps can be seen as factorizations of the identity into $4$ permutations. We now consider factorizations of the identity into an arbitrary number of permutations. For $m \geq 2$, let $\mathcal{M}_{m}(n, \lambda)$ be the set of factorizations $\phi \sigma_0 \sigma_1 \cdots \sigma_{m-1} = \id_n$ in $\mathbb{S}_n$ such that $\phi \in C(\lambda)$ and $\sigma_0 \in C([n])$. The genus of such a tuple $(\phi, \sigma_0, \ldots, \sigma_{m-1}) \in \mathcal{M}_m(n,\lambda)$ is defined by
\begin{equation} \label{eq:const-genus-def}
  g(\phi, \sigma_0, \ldots, \sigma_{m-1}) = 1 + \frac{1}{2} \left( (m-1)n - \ell(\lambda) - \sum_{i=0}^{m-1} \cyc(\sigma_i) \right).
\end{equation}
We note that $\mathcal{M}_{m}(n,\lambda)$ consists of rotation systems for $m$-constellations with one hyperface and the degrees of vertices of the first color given by $\lambda$ (\textit{cf.} \cite[Chapter~1]{LandoZvonkine}).

By writing $\id_n = \phi_\circ \sigma_\circ \phi_\bullet \sigma_\bullet = \sigma_\bullet \sigma_\circ \left(\sigma_\circ^{-1} \phi_\circ \sigma_\circ\right) \phi_\bullet$, we see that rotation systems for bi-Eulerian maps are in genus-preserving bijection with factorizations in $\mathcal{M}_{3}(n,\lambda)$. We define the genus polynomial of $\mathcal{M}_{m}(n,\lambda)$ to be
\begin{equation}
  \label{eq:genus-poly-const}
  \Gamma^{(m)}_{n,\lambda}(x) = \sum_{M \in \mathcal{M}_m(n,\lambda)} x^{g(M)}.
\end{equation}
By \eqref{eq:gpoly-fact}, we have
\begin{equation} \label{eq:gpoly-gamma}
  \Gamma_{n,\lambda}(x) = \frac1{n!} \left(\prod_{i=1}^{\ell(\lambda)} \lambda_i!\right)^2 \cdot \left( \prod_{i \geq 1} m_i(\lambda)! \right) \cdot \Gamma^{(3)}_{n,\lambda}(x).
\end{equation}

\begin{rmk}
Bi-Eulerian maps were already considered in \cite{unicursal}, which also contained an enumeration formula of planar bi-Eulerian maps. In \cite{bipartite-eulerian}, a bijection between planar bi-Eulerian maps and another family of maps called \emph{planar 3-constellations} was given, which actually works for all genera. Such a connection subsumed the enumeration formula in \cite{unicursal} under the framework of constellations, with the planar case solved in \cite{BMS} by Bousquet-Mélou and Schaeffer.
\end{rmk}

\section{Explicit formula for genus polynomial via characters} \label{sec:chara}

To simplify the notations, we define the polynomial $P^{(m)}_{n, \lambda}(t)$ by
\begin{equation}
  \label{eq:p-def}
  P^{(m)}_{n,\lambda}(t) = t^{(m-1)n+1} \Gamma^{(m)}_{n,\lambda}(t^{-2}).
\end{equation}
In the following, we will mainly work with $P^{(m)}_{n, \lambda}(t)$.

For $\theta, \mu$ two partitions of $n$, we denote by $\chi^\theta_\mu$ the irreducible character of $\mathbb{S}_n$ indexed by $\theta$ and evaluated at the conjugacy class $C(\mu)$. The following formula for the number of factorizations of the identity with given cycle type for each factor is often attributed to Frobenius:
\begin{equation} \label{eq:frobenius}
  \left| \{ \sigma_1 \cdots \sigma_k = \id_n \mid \forall i, \sigma_i \in C(\mu^{(i)}) \} \right| = \frac{1}{n!} \prod_{i=1}^{k} \left| C(\mu^{(i)}) \right| \sum_{\theta \vdash n} \left(f^{\theta}\right)^{2-k} \prod_{i=1}^{k} \chi^\theta_{\mu^{(i)}}.
\end{equation}
Here, $f^\theta = \chi^\theta_{[1^n]}$ is the dimension of the representation of $\mathbb{S}_n$ indexed by $\theta$. This formula can be proved using the orthogonal idempotent basis in the center of the group algebra of $\mathbb{S}_n$ (\textit{cf.} \cite{serre1977linear}).

To express $\Gamma^{(m)}_{n,\lambda}$ with \eqref{eq:frobenius}, we need the following lemma. The \tdef{Ferrers diagram} (with the \emph{French convention}) of a partition $\theta = (\theta_1, \ldots, \theta_k)$ is formed by $k$ left-aligned rows of unit squares called \tdef{cells}, with $\theta_i$ cells for the $i$-th rows from bottom up. The coordinate of the cell at the lower-left corner is $(0,0)$. The \tdef{content} of a cell $w$ at coordinate $(i,j)$ is $c(w) = i - j$.

\begin{lem}[Lemma~3.4 in \cite{JV1990a}, see also Lemma~1 in \cite{fang2014generalization}] \label{lem:poly-H}
  Given $\theta \vdash n$, we have
  \begin{equation} \label{eq:h-poly}
    \sum_{\mu \vdash n} \left| C(\mu) \right| \chi^\theta_\mu x^{\ell(\mu)} = f^\theta H_\theta(x).
  \end{equation}
  Here, $H_\theta(x) = \prod_{w} (x+c(w))$, with $w$ running over cells in the Ferrers diagram of $\theta$.
\end{lem}

A standalone proof using the group algebra of $\mathbb{S}_n$ can be found in \cite{fang2014generalization}. We would also like to mention some special character values in $\mathbb{S}_n$. For $0 \leq a \leq n-1$, we denote by $\hook(n,a)$ the \tdef{hook partition} $[n-a, 1^a]$. For $a=n-1$, we have $\hook(n,n-1) = [1^n]$. The following result is a well-known direct consequence of the Murnaghan-Nakayama rule (\textit{cf.} \cite[Chapter~7.17]{Stanley:EC2}.

\begin{prop} \label{prop:mn-cycle}
  For any $\theta \vdash n$, we have $\chi^\theta_{[n]} = 0$ unless $\theta = \hook(n,a)$ for some $a$, and we have $\chi^{\hook(n,a)}_{[n]} = (-1)^{a}$.
\end{prop}

Combining the ingredients above, we have the following expression of $P^{(m)}_{n,\lambda}(t)$.

\begin{thm} \label{thm:p-expr}
  For $m \geq 2$, $n \geq 1$ and $\lambda \vdash n$, we have
  \[
    P^{(m)}_{n,\lambda}(t) = \frac{\left| C(\lambda) \right|}{n} t^{\ell(\lambda)} \sum_{a=0}^{n-1} (-1)^{a} H_{n,a}(t)^{m-1} \chi^{\hook(n,a)}_\lambda.
  \]
  Here, $H_{n,a}(t) = H_{\hook(n,a)}(t) = \prod_{k=-a}^{n-a-1} (t+k)$.
\end{thm}
\begin{proof}
  By the definitions \eqref{eq:const-genus-def}, \eqref{eq:genus-poly-const} and \eqref{eq:p-def}, we have
  \[
    P^{(m)}_{n,\lambda}(t) = \frac{\left| C(\lambda) \right|}{n} \sum_{\mu^{(1)}, \ldots, \mu^{(m-1)} \vdash n} \prod_{i=1}^{m-1} \left| C(\mu^{(i)}) \right| \sum_{\theta \vdash n} \left(f^{\theta}\right)^{1-m} \chi^\theta_{[n]} \chi^\theta_\lambda \prod_{i=1}^{m-1} \chi^\theta_{\mu^{(i)}} t^{\ell(\lambda) + \sum_{i=1}^{m-1} \ell(\mu^{(i)})}.
  \]
  Then by Lemma~\ref{lem:poly-H}, we have
  \[
    P^{(m)}_{n,\lambda}(t) = \frac{\left| C(\lambda) \right|}{n} t^{\ell(\lambda)} \sum_{\theta \vdash n} H_\theta(t)^{m-1} \chi^\theta_{[n]} \chi^\theta_\lambda.
  \]
  Combining with Proposition~\ref{prop:mn-cycle} and the definition of $H_\theta(t)$ in Lemma~\ref{lem:poly-H}, we establish our claim.
\end{proof}

\begin{rmk}
  Factorization of the identity in $\mathbb{S}_n$ has been studied abundantly in the literature. In the jargon of map enumeration, such factorizations into $m+1$ permutations correspond to \emph{$m$-constellations} (\textit{cf.} \cite[Chapter~1]{LandoZvonkine}). The case where one of the permutation is the $n$-cycle (the \emph{one-face} case in terms of maps) is particularly well-studied. Jackson first gave in \cite{jackson-factorization} an explicit formula for such factorizations with control on the cycle types of all factors, using characters of $\mathbb{S}_n$. Poulalhon and Schaeffer then gave in \cite{poulalhon2002factorizations} a formula without cancellation using bijective methods, while also giving some asymptotic enumeration results. Bernardi and Morales also studied in \cite{bernardi-morales} symmetries in such factorizations using bijections.
\end{rmk}

Using Theorem~\ref{thm:p-expr}, we can compute the genus polynomials for some special digraphs. The \tdef{directed bouquet} of order $n$, denote by $B_n$, is the digraph formed by one single vertex and $n$ loops. We now compute the genus distribution of $B_n$.

\begin{coro} \label{coro:bouquet}
  The genus polynomial $\Gamma_{B_n}(x)$ of the directed bouquet $B_n$ is
  \[
    \Gamma_{B_n}(x) = \frac1{n} x^{(n+1)/2} \sum_{a=0}^{n-1} (-1)^a \binom{n-1}{a} H_{n,a}(x^{-1/2})^{m-1}.
  \]
\end{coro}
\begin{proof}
  By adding a white vertex in the middle of each loop in $B_n$, we obtain the digraph $D_{n,[1^n]}$. This reversible transformation can be done on every face-oriented embedding of $B_n$ without changing the genus, thus $\Gamma_{B_n}(x) = \Gamma_{n, [1^n]}(x)$. We conclude by combining Theorem~\ref{thm:p-expr} with \eqref{eq:gpoly-gamma}, \eqref{eq:p-def}, and the fact that $\chi^{\hook(n,a)}_{[1^n]} = \binom{n-1}{a}$.
\end{proof}

We note that, although there seems to be $1/n$ and fractional powers of $x$ appearing in Corollary~\ref{coro:bouquet}, $\Gamma_{B_n}(x)$ is still a polynomial in $x$ with integral coefficients due to cancellations. Table~\ref{tab:bouquet} lists the genus polynomials of $B_n$ for $1 \leq n \leq 8$.

\begin{table}[!thb]
  \centering
  \begin{tabular}{cl}
    \toprule
    $n$ & $\Gamma_{B_n}(x)$ \\
    \midrule
    $1$ & $1$\\
    $2$ & $2$\\
    $3$ & $10 + 2 x$\\
    $4$ & $84 + 60 x$\\
    $5$ & $1008 + 1680 x + 192 x^2$\\
    $6$ & $15840 + 50400 x + 20160 x^2$ \\
    $7$ & $308880 + 1663200 x + 1527120 x^2 + 129600 x^3$ \\
    $8$ & $7207200 + 60540480 x + 104781600 x^2 + 30683520 x^3$ \\
    \bottomrule
  \end{tabular}
  \caption{Genus polynomials for the directed bouquets $B_{n}$, $1 \leq n \leq 8$} \label{tab:bouquet}
\end{table}

Another special case we consider here is the \tdef{directed dipole} $D_{n,[n]}$ of order $n$, which is the Eulerian digraph with two vertices and $2n$ edges. We now compute the genus polynomial of $D_{n,[n]}$.

\begin{coro} \label{coro:dipole}
  The genus polynomial $\Gamma_{n,[n]}(x)$ of the directed dipole $D_{n,[n]}$ is
  \[
    \Gamma_{n,[n]}(x) = \left( (n-1)! \right)^2 x^n \sum_{a=0}^{n-1} H_{n,a}(x^{-1/2})^2.
  \]
\end{coro}
\begin{proof}
  The computation is the same as that of Corollary~\ref{coro:bouquet}, except that we use $\chi^{\hook(n,a)}_{[n]} = (-1)^a$ in Proposition~\ref{prop:mn-cycle} here.
\end{proof}

Using Corollary~\ref{coro:dipole}, we list the genus polynomials of the directed dipole $D_{n,[n]}$ for $1 \leq n \leq 6$ in Table~\ref{tab:dipole}.

\begin{table}[!thb]
  \centering
  \begin{tabular}{cl}
    \toprule
    $n$ & $\Gamma_{n,[n]}(x)$ \\
    \midrule
    $1$ & $1$ \\
    $2$ & $2 + 2x$ \\
    $3$ & $12 + 96x + 36x^2$ \\
    $4$ & $144 + 4320x + 13392x^2 + 2880x^3$ \\
    $5$ & $2880 + 230400x + 2594880x^2 + 4752000x^3 + 714240x^4$ \\
    $6$ & $86400 + 15120000x + 440899200x^2 + 2867184000x^3 + 3706214400x^4 + 435456000x^5$ \\
    \bottomrule
  \end{tabular}
  \caption{Genus polynomials for the directed dipoles $D_{n,[n]}$, $1 \leq n \leq 6$}
  \label{tab:dipole}
\end{table}

\section{Backward shift operator and real-rootedness} \label{sec:realrooted}

To prove the real-rootedness of $\Gamma^{(m)}_{n,\lambda}$, we follow the steps in \cite{stanley-factorization}. Exceptionally in this section, the symbol $i$ stands for the imaginary unity instead of an integer. Let $\mathbf{T}$ be the \tdef{backward shift operator} on the polynomial space $\mathbb{C}[t]$ defined by $\mathbf{T}Q(t) = Q(t-1)$. We denote by $(n)_k$ the falling factorial $n(n-1)\cdots(n-k+1)$. We have the following minor generalization of Theorem~3.2 in \cite{stanley-factorization}.

\begin{prop} \label{prop:shift}
  Suppose that $F \in \mathbb{C}[z] \setminus \{0\}$ is of degree $d$ and with all zeros of modulus $1$. Let $c$ be the multiplicity of $1$ as a root of $F(z)$. For any given $k \in \mathbb{N}_+$, we take $P(t) = F(\mathbf{T}) (t+n-1)_n^k$.
  \begin{itemize}
  \item If $d \leq n-1$, then $P(t) = (t+n-d-1)_{n-d}^k Q(t)$ with $Q(t)$ a polynomial of degree $kd-c$ whose roots all have the real part $(d-n+1)/2$;
  \item If $d \geq n-1$, then $P(t)$ is a polynomial of degree $kn-c$ whose roots all have the real part $(d-n+1)/2$.
  \end{itemize}
\end{prop}
\begin{proof}
  We first discuss the degrees. We observe that, for $k \geq 1$, we have $\mathbf{T}t^k = t^k + R(t)$ with $R(t)$ of degree $k-1$. Thus, for any polynomial $R \in \mathbb{C}[t]$ we have $\deg((\mathbf{T} - 1)R) = \deg(R) - 1$, while $\deg((\mathbf{T} - w)R) = \deg(R)$ for all $w \neq 1$. We check easily the degrees of $P(t)$ and $Q(t)$.

  We now proceed by induction on $d$. The case $d=0$ is trivial. Suppose that our claim holds for some $d < n - 1$. Then, for any $F(z)$ with degree $d$ and any $w \in \mathbb{C}$ with $|w|=1$, let $Q(t)$ be the polynomial in the induction hypothesis, and we have
  \begin{align*}
    (\mathbf{T} - w)F(\mathbf{T}) (t+n-1)_n^k &= (t+n-d-1)_{n-d}^k Q(t) - w(t+n-d-2)_{n-d}^k Q(t-1) \\
    &= (t+n-d-2)_{n-d-1}^k \left( (t+n-d-1)^k Q(t) - w(t-1)^kQ(t-1) \right).
  \end{align*}

  Let $Q_+(t) = (t+n-d-1)^k Q(t) - w(t-1)^kQ(t-1)$. To prove that $Q_+(t)$ has our claimed properties, we slightly adapt arguments in Lemma~9.13 in \cite{postnikov-stanley}. By induction hypothesis, we can write
  \[
    Q(t) = \prod_{j = 1}^{kd-c} \left( t - \frac{d-n+1}{2} + i \delta_j \right)
  \]
  for some real numbers $\delta_j$. Let $t_0$ be a zero of $Q_+(t)$. By writing $t_0 = (d-n+2)/2 + \alpha + i\beta$ for $\alpha, \beta \in \mathbb{R}$, since $|w| = 1$, we have
  \[
    \left| \frac{n-d}{2} + \alpha + i\beta \right|^k \prod_{j = 1}^{kd-c} \left| \frac{1}{2} + \alpha + i (\delta_j+\beta) \right| = \left| \frac{d-n}{2} + \alpha + i\beta \right|^k \prod_{j = 1}^{kd-c} \left| - \frac{1}{2} + \alpha + i (\delta_j+\beta) \right|.
  \]
  If $\alpha > 0$, since $d < n - 1$, we have both $|(n-d)/2 + \alpha + i\beta| > |(d-n)/2 + \alpha + i\beta|$ and $|1/2 + \alpha + i(\delta_j + \beta)| > |-1/2 + \alpha + i(\delta_j + \beta)|$. Therefore, the equality above cannot hold. The case $\alpha < 0$ is similar, with the two inequalities above reversed. Therefore, we have $\alpha = 0$, meaning that any zero of $Q_+(t)$ must have real part $(d-n+2)/2$. We thus complete the induction step from $d$ to $d+1$ for $d < n-1$. The induction for $d \geq n-1$ is essentially the same, and can be viewed as taking $k=0$ in $Q_+(t)$.
\end{proof}

To show that $P_{n,\lambda}^{(m)}$ is real-rooted, we need the following summation on characters of the form $\chi^{\hook(n,a)}_{\lambda}$, whose proof can be found in \cite{zagier1995distribution}.

\begin{prop}[Equation~(9) in \cite{zagier1995distribution}] \label{prop:chara-sum}
  For $n > 0$ and $\lambda = (\lambda_1, \lambda_2, \ldots) \vdash n$, we have
  \[
    \sum_{a=0}^{n-1} (-1)^a \chi^{\hook(n,a)}_\lambda z^a = \frac{1}{1-z} \prod_{j=1}^{\ell(\lambda)} (1-z^{\lambda_j}).
  \]
\end{prop}

We now reach our main result on real-rootedness.

\begin{thm} \label{thm:main}
  For any intergers $n > 0$ and $m \geq 2$, for any partition $\lambda \vdash n$, all the roots of $\Gamma^{(m)}_{n,\lambda}(x)$ are real and non-positive.
\end{thm}
\begin{proof}
  We first observe that, for $0 \leq a \leq n-1$, we have $H_{n,a}(t) = \mathbf{T}^a H_{n,0}(t)$. Let $R_{n,\lambda}(z) = \frac{1}{1-z} \prod_{j=1}^{\ell(\lambda)} (1-z^{\lambda_j})$. By Theorem~\ref{thm:p-expr}~and~\ref{prop:chara-sum}, we have
  \begin{equation} \label{eq:op-expr}
    P^{(m)}_{n,\lambda}(t) = \frac{\left| C(\lambda) \right|}{n} t^{\ell(\lambda)} R_{n,\lambda}(\mathbf{T}) H_{n,0}(t)^{m-1}.
  \end{equation}
  It is clear that $R_{n,\lambda}(z)$ is a polynomial of degree $n-1$ with all zeros of modulus $1$. Furthermore, $1$ is a root of multiplicity $\ell(\lambda)-1$ in $R_{n,\lambda}(z)$. We also observe that $H_{n,0}(t) = (t+n-1)_n$. By Proposition~\ref{prop:shift}, the roots of $P^{(m)}_{n,\lambda}(t)$ are all in $i\mathbb{R}$. By \eqref{eq:p-def}, all the roots of $\Gamma^{(m)}_{n,\lambda}(x)$ are real and non-positive.
\end{proof}

\begin{coro} \label{coro:dipole}
  The genus distribution of $D_{n,\lambda}$ is log-concave.
\end{coro}
\begin{proof}
  It suffices to start from \eqref{eq:gpoly-gamma}, then take $m=3$ in Theorem~\ref{thm:main} and apply Lemma~\ref{lem:real-rooted}.
\end{proof}

\section{Asymptotic genus distributions} \label{sec:asympt}

We now try to compute the asymptotic genus distribution for some special cases of $\mathcal{M}_m(n,\lambda)$ when $n \to +\infty$, using the real-rootedness of $\Gamma^{(m)}_{n,\lambda}(x)$ and Lemma~\ref{lem:rr-dist}. We only need to compute the mean and the variance of genus. We start by some computational results.

\begin{prop} \label{prop:simple-op}
  For $n \geq 0$, $m \geq 3$ and $1 \leq k \leq n-1$, we have
  \begin{equation} \label{eq:diff1}
    \left. \mathbf{T}^k H_{n,0}(t)^{m-1} \right|_{t=1} = \left. \frac{d}{dt} \mathbf{T}^k H_{n,0}(t)^{m-1} \right|_{t=1} = 0.
  \end{equation}
  Furthermore, for $m \geq 4$, we have
  \begin{equation} \label{eq:diff2-large}
    \left. \frac{d^2}{dt^2} \mathbf{T}^k H_{n,0}(t)^{m-1} \right|_{t=1} = 0.
  \end{equation}
  We also have, for $1 \leq k \leq n-1$,
  \begin{equation} \label{eq:diff2-small}
    \left. \frac{d^2}{dt^2} \mathbf{T}^k H_{n,0}(t)^{2} \right|_{t=1} = 2 \left( (n-k)! (k-1)! \right)^2.
  \end{equation}
\end{prop}
\begin{proof}
  We observe that $H_{n,0}(t) = (t+n-1)_n$, meaning that $\mathbf{T}^k H_{n,0}(t)^{m-1} = (t-k+n-1)_n^{m-1}$ contains a factor $(t-1)^{m-1}$. Therefore, both $\mathbf{T}^k H_{n,0}(t)^{m-1}$ and $\frac{d}{dt} \mathbf{T}^k H_{n,0}(t)^{m-1}$ contain a factor $(t-1)^{m-2}$. When $m \geq 3$, an evaluation at $t=1$ gives $0$. The second equality can be obtained similarly.

  For the last equality, we write $\mathbf{T}^k H_{n,0}(t)^2 = (t-1)^2 R(t)$, with $R(t) = (t-k+n-1)_{n-k}^2 (t-2)_{k-1}^2$. The left-hand side of the third equality is thus equal to $2R(1)$.
\end{proof}

Given $m \geq 3$, $n > 0$ and $\lambda \vdash n$, let $M$ be a random element of $\mathcal{M}_m(n,\lambda)$ chosen uniformly. We consider the random variable $X_{n,\lambda}^{(m)} = (mn+1) - 2g(M)$. From \eqref{eq:p-def} we see that the characteristic polynomial of $X_{n,\lambda}^{(m)}$ is $P_{n,\lambda}^{(m)}(t) / P_{n,\lambda}^{(m)}(1)$. We now study the expectation and variance of $X_{n,\lambda}^{(m)}$. We use the polynomial $R_{n,\lambda}(z) = (1-z)^{-1} \prod_{j=1}^{\ell(\lambda)} (1-z^{\lambda_j})$ as defined in the proof of Theorem~\ref{thm:main}.

\begin{prop} \label{prop:expectation}
  For any integer $n > 0$ and $m \geq 3$, and for any partition $\lambda \vdash n$, we have
  \[
    \mathbb{E}[X_{n,\lambda}^{(m)}] = \ell(\lambda) + (m-1) \sum_{k=1}^{n} \frac{1}{k}.
  \]
\end{prop}
\begin{proof}
  We first observe that
  \[
    \mathbb{E}[X_{n,\lambda}^{(m)}] = \frac1{P_{n,\lambda}^{(m)}(1)} \cdot \left. \frac{d}{dt} P_{n,\lambda}^{(m)}(t) \right|_{t=1}.
  \]
  Now we only need to compute $\left. \frac{d}{dt} P_{n,\lambda}^{(m)}(t) \right|_{t=1}$ and $P_{n,\lambda}^{(m)}(1)$. We first have the following equality:
  \begin{align}
    \begin{split} \label{eq:diff-H-eval}
      \frac{d}{dt} H_{n,0}(t)^{m-1} &= \sum_{k=0}^{n-1} (m-1)(t+k)^{m-2} \prod_{0 \leq j \leq n-1, j \neq k} (t+j)^{m-1} \\
      &= (m-1) H_{n,0}(t)^{m-1} \sum_{k=0}^{n-1} \frac1{t+k}.
    \end{split}
  \end{align}
  We also notice that $R_{n,\lambda}(z)$ is a polynomial in $z$ of degree $n-1$, and $[z^0]R_{n,\lambda}(z) = 1$. Then, by differentiating \eqref{eq:op-expr} and then plugging in \eqref{eq:diff1} and \eqref{eq:diff-H-eval}, we have
  \begin{align*}
    \left. \frac{d}{dx} P_{n,\lambda}^{(m)}(x) \right|_{x=1} &=  \frac{|C(\lambda)|}{n} \left( \ell(\lambda) (n!)^{m-1} + \left. \frac{d}{dt} H_{n,0}(t)^{m-1} \right|_{t=1} \right) \\
    &= \frac{|C(\lambda)|}{n} (n!)^{m-1} \left( \ell(\lambda) + (m-1) \sum_{k=1}^{n} \frac{1}{k} \right).
  \end{align*}
  To compute $P_{n,\lambda}^{(m)}(1)$, we also combine \eqref{eq:op-expr} and \eqref{eq:diff1} to have
  \begin{equation} \label{eq:normalizer}
    P_{n,\lambda}^{(m)}(1) = \frac{|C(\lambda)|}{n} (n!)^{m-1}.    
  \end{equation}
  We thus have our claim.
\end{proof}

\begin{prop} \label{prop:variance}
  For any integer $n > 0$ and $m \geq 4$, and for any partition $\lambda \vdash n$, we have
  \[
    \mathbf{Var}[X_{n,\lambda}^{(m)}] = (m-1) \sum_{k=1}^n \left( \frac1{k} - \frac1{k^2} \right).
  \]
  For the case $m = 3$, we have
  \[
    \mathbf{Var}[X_{n,\lambda}^{(3)}] = 2 \sum_{k=1}^n \left( \frac1{k} - \frac1{k^2} \right) + 2n^{-2} \sum_{k=1}^{n-1} \binom{n-1}{k-1}^{-2} [z^k]R_{n,\lambda}(z).
  \]
\end{prop}
\begin{proof}
  We first observe that
  \begin{equation} \label{eq:variance}
    \mathbf{Var}[X_{n,\lambda}^{(m)}] = \frac1{P_{n,\lambda}^{(m)}(1)} \cdot \left. \frac{d^2}{dt^2} P_{n,\lambda}^{(m)} \right|_{t=1} + \mathbb{E}[X_{n,\lambda}^{(m)}] - \mathbb{E}[X_{n,\lambda}^{(m)}]^2.
  \end{equation}
  We thus only need to compute $\left. \frac{d^2}{dt^2} P_{n,\lambda}^{(m)} \right|_{t=1}$. We will also use the fact that $[z^0]R_{n,\lambda}(z) = 1$. By differentiating \eqref{eq:op-expr} while simplifying with \eqref{eq:diff1}, we have
  \begin{align*}
    \left. \frac{d^2}{dt^2} P_{n,\lambda}^{(m)} \right|_{t=1} = \frac{|C(\lambda)|}{n} \bigg( \ell(\lambda)(\ell(\lambda)-1) (n!)^{m-1} + &2\ell(\lambda) \left. \frac{d}{dt} H_{n,0}(t)^{m-1}\right|_{t=1} \\
                                                                                                                                          &+ \left. \frac{d^2}{dt^2} R_{n,\lambda} (\mathbf{T}) H_{n,0}(t)^{m-1}\right|_{t=1} \bigg)
  \end{align*}
  By substituting the equality above and the result in Proposition~\ref{prop:expectation} into \eqref{eq:variance}, using \eqref{eq:diff1} and \eqref{eq:diff-H-eval}, we compute the variance of $X_{n,\lambda}^{(m)}$ as
  \begin{equation} \label{eq:variance-mid}
    \mathbf{Var}[X_{n,\lambda}^{(m)}] = (m-1) \sum_{k=1}^n \frac1{k} - (m-1)^2 \left( \sum_{k=1}^n \frac1{k} \right)^2 + (n!)^{1-m} \left. \frac{d^2}{dt^2} R_{n,\lambda} (\mathbf{T}) H_{n,0}(t)^{m-1}\right|_{t=1}
  \end{equation}
  We observe that $[z^0]R_{n,\lambda}(z) = 1$, which leads us to compute the differentiation of \eqref{eq:diff-H-eval} while simplifying with \eqref{eq:diff-H-eval} itself as
  \begin{equation} \label{eq:diff-H-eval-2}
    \frac{d^2}{dt^2} H_{n,0}(t)^{m-1} = (m-1) H_{n,0}(t)^{m-1} \left( (m-1) \left( \sum_{k=0}^{n-1} \frac1{t+k} \right)^2 - \sum_{k=0}^{n-1} \frac1{(t+k)^2} \right).
  \end{equation}

  For the case $m \geq 4$, we use \eqref{eq:diff2-large} and \eqref{eq:diff-H-eval-2} to simplify \eqref{eq:variance-mid} to obtain our claim. For the case $m = 3$, it suffices to replace \eqref{eq:diff2-large} with \eqref{eq:diff2-small} in the computation.
\end{proof}

We can now compute the asymptotic genus distribution of $\mathcal{M}_m(n,\lambda)$.

\begin{thm} \label{thm:asympt-dist}
  Given $m \geq 3$ and a sequence of partitions $(\lambda^{(n)})_{n \geq 1}$ with $\lambda^{(n)} \vdash n$, for an element $M$ chosen uniformly from $\mathcal{M}_m(n,\lambda^{(n)})$, we have
  \[
    \frac{g(M) - \mu_m}{\sigma_m} \underset{n \to \infty}{\longrightarrow} \mathcal{N}(0,1),
  \]
  with
  \[
    \mu_m = \frac{(m-1)n-\ell(\lambda^{(n)})}{2} - \frac{m-1}{2} \ln(n), \quad \sigma_m^2 = \frac{m-1}{4} \ln(n).
  \]
\end{thm}
\begin{proof}
  For $m \geq 4$, we obtain our claim by combining Proposition~\ref{prop:expectation} and Proposition~\ref{prop:variance} with Lemma~\ref{lem:rr-dist} and Theorem~\ref{thm:main}, and also the asymptotic value of the harmonic sum.

  For the case $m = 3$, we start by bounding $[z^k]R_{n,\lambda}(z)$ using its definition:
  \begin{align*}
    \left| [z^k]R_{n,\lambda}(z) \right| &\leq [z^k] (1+z+\cdots+z^{\lambda_1-1}) \prod_{i=2}^{\ell(\lambda)} (1+z^{\lambda_i}) \\
                                         &\leq [z^k] (1+z)^{\lambda_1-1} \prod_{i=2}^{\ell(\lambda)} (1+z)^{\lambda_i} = [z^k](1+z)^{n-1} = \binom{n-1}{k}.
  \end{align*}
  Also observing $[z^{n-1}]R_{n,\lambda}(z) = 1$, we have the following bound:
  \begin{align*}
    \left| 2n^{-2} \sum_{k=1}^{n-1} \binom{n-1}{k-1}^{-2} [z^k]R_{n,\lambda}(z) \right| &\leq 2n^{-2} \sum_{k=1}^{n-1} \binom{n-1}{k-1}^{-2} \binom{n-1}{k} \\
    &\leq 2n^{-1} \sum_{k=1}^{n-1} \binom{n-1}{k-1}^{-1} \leq 2.
  \end{align*}
  Therefore, when $n \to \infty$, this extra term is negligible comparing with $\sigma_m^2$. The rest is the same as for $m \geq 4$.
\end{proof}

\begin{rmk}
  The asymptotic distribution in Theorem~\ref{thm:asympt-dist} is to be expected. In fact, by applying standard techniques in analytic combinatorics (see \cite[Chapter~IX.7.1]{flajolet}) on the Harer-Zagier formula (Equation~(4) in \cite{harer-zagier}) counting one-face maps, we can obtain similar asymptotic distribution (see \cite{chord-asym} for a stronger local limit result), and such a distribution should be typical for many families of maps.
\end{rmk}

\section{Genus distribution of Eulerian fans via transformation} \label{sec:other}

We now propose another family of digraphs whose genus distributions can be obtained via that of $D_{n,\lambda}$. A \tdef{directed forest} is a digraph whose underlying undirected graph is a forest. Given an Eulerian digraph $D$, if there is a vertex $u_\bullet$ with equal in-degree and out-degree such that its removal (including all its edges) leaves a directed forest, then we say that $D$ is an \tdef{Eulerian fan}, with $u_\bullet$ its \tdef{handle}. An example of an Eulerian fan is given in Figure~\ref{fig:eulerian-fan-ex}. Its genus polynomial is
\[
  160 + 16000 x + 249120 x^2 + 766400 x^3 + 350720 x^4.
\]

\begin{figure}[!thbp]
  \centering
  \includegraphics[page=4,width=0.2\textwidth]{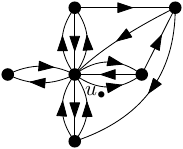}
  \caption{An example of Eulerian fan}
  \label{fig:eulerian-fan-ex}
\end{figure}

For $M$ a face-oriented embedding of an Eulerian fan $D$ with $u_\bullet$ as handle, its \tdef{skeleton}, denoted by $\sk(M)$, is the directed forest obtained by deleting $u_\bullet$ and its edges from $D$, decorated by \tdef{blossoms} standing for the deleted edges (with labels kept), indicating how such edges are ordered in the embedding $M$. For instance, Figure~\ref{fig:eulerian-fan}(b) is the skeleton of Figure~\ref{fig:eulerian-fan}(a).

\begin{figure}[!thbp]
  \centering
  \includegraphics[page=3,width=0.5\textwidth]{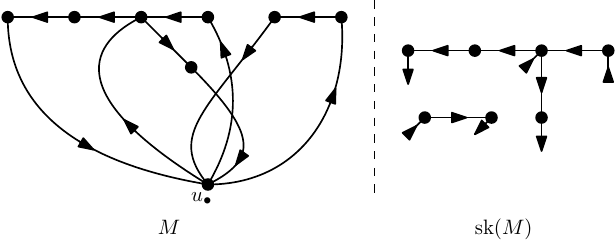}
  \caption{A face-oriented embedding of an Eulerian fan, and its skeleton (labels are omitted)}
  \label{fig:eulerian-fan}
\end{figure}

\begin{prop} \label{prop:eulerian-fan}
  The genus distribution of an Eulerian fan $D$ is log-concave.
\end{prop}
\begin{proof}
  Let $\Gamma_D(x)$ be the genus polynomial of a labeled version of $D$, and $u_\bullet$ the handle of $D$. Let $n$ be the in-degree of $u_\bullet$, and $F$ the forest obtained by deleting $u_\bullet$ and its edges. We can suppose that the labeling of $D$ is such that the $n$ out-going edges of $u_\bullet$ are labeled from $1$ to $n$, and the $n$ in-coming ones from $n+1$ to $2n$. We denote by $\lambda \vdash n$ the partition formed by the number of in-coming edges from each component of $F$ to $u_\bullet$. We denote by $\mathcal{S}(D)$ the set of all the possible skeletons of face-oriented embeddings of $D$.

  Each component of $F$ in $M$ can be contracted into a white vertex with equal in-degree and out-degree. We thus obtain a face-oriented embedding of $D_{n,\lambda}$, denoted by $\ct(M)$, with labels determined by those in $M$. However, we notice that the cyclic order of edges around white vertices of $\ct(M)$ is the same as that of $\sk(M)$. Therefore, we have a $|\mathcal{S}(D)|$-to-$\left(\prod_{i}(\lambda_i-1)!\lambda_i!\right)$ correspondence between $\mathcal{M}(D)$ and $\mathcal{M}(D_{n,\lambda})$. Furthermore, since $F$ is a forest, we have $g(\ct(M)) = g(M)$. We thus have
  \[
    \Gamma_D(x) = \frac1{| \mathcal{S}(D) |} \left(\prod_{i=1}^{\ell(\lambda)} (\lambda_i-1)! \lambda_i! \right) \Gamma_{n,\lambda}(x).
  \]
  By Theorem~\ref{thm:main}, $\Gamma_D(x)$ is real-rooted with non-positive roots. Therefore, the genus distribution of $D$ is log-concave by Lemma~\ref{lem:real-rooted}.
\end{proof} 

\bibliographystyle{alpha}
\bibliography{graph-genus-dist}

\end{document}